\documentclass[11pt]{article}
\usepackage{xcolor}
\usepackage{amsthm}
\usepackage{amsfonts}
\usepackage{graphicx}
\usepackage{subfig}
\usepackage{latexsym}
\usepackage{url}
\usepackage{bbm}
\usepackage{mathtools}
\mathtoolsset{showonlyrefs}

\newcommand{\CC}{\mathbb{C}}
\newcommand{\R}{\mathbb{R}}

\newcommand{\N}{\mathbb{N}}
\newcommand{\Z}{\mathbb{Z}}

\newcommand{\1}{\mathbbm{1}}

\renewcommand{\P}{\mathbb P}

\newcommand{\RE}{\textrm{Re}}

\newcommand{\al}{\alpha}

\newcommand{\ga}{\gamma}

\newcommand{\ep}{\varepsilon}

\newcommand{\de}{\delta}
\newcommand{\te}{\theta}

\newcommand{\toop}{\stackrel{\P}{\longrightarrow}}
\newcommand{\schw}{\stackrel{\raisebox{-1pt}{\textup{\tiny d}}}{\longrightarrow}}

\newcommand{\eqschw}{\stackrel{d}{=}}

\newcommand{\bee}{\begin{equation}}
\newcommand{\eee}{\end{equation}}
\newcommand{\beea}{\begin{array}}
\newcommand{\eeea}{\end{array}}

\renewcommand{\theequation}{\arabic{section}.\arabic{equation}}

\theoremstyle{plain}
\newtheorem{prop}{Proposition}[section]

\newtheorem{theo}[prop]{Theorem}
\newtheorem{lem}[prop]{Lemma}

\theoremstyle{definition}

\newtheorem{rem}[prop]{Remark}

\begin{document}

\title{Asymptotic theory for quadratic variation of \\ harmonizable fractional stable processes} 

\author{Andreas Basse-O'Connor\thanks{Department
of Mathematics, University of Aarhus, 
E-mail: basse@math.au.dk.}  \and
Mark Podolskij\thanks{Department
of Mathematics, University of Luxembourg,
E-mail: mark.podolskij@uni.lu.}
\thanks{Mark Podolskij
gratefully acknowledges financial support of ERC Consolidator Grant 815703
``STAMFORD: Statistical Methods for High Dimensional Diffusions''.}}

\maketitle

\begin{abstract}
\noindent In this paper we study the asymptotic theory for quadratic variation of a   
harmonizable fractional $\al$-stable process. We show a law of large numbers with a non-ergodic limit and obtain weak convergence towards a L\'evy-driven Rosenblatt random variable when the Hurst parameter satisfies $H\in (1/2,1)$ and $\al(1-H)<1/2$. This result complements the asymptotic theory for fractional stable processes investigated in e.g.\ \cite{BHP19,BLP17,BP17,BPT20,LP18,MOP20}.

\ \

\noindent
{\it Key words}: \
fractional processes, harmonizable processes, limit theorems, quadratic variation, stable L\'evy motion.
\bigskip

\noindent
{\it AMS 2010 subject classifications}: 60F05, 60F15, 60G22, 60G48, 60H05. 

\end{abstract}

\section{Introduction} \label{sec1}
\setcounter{equation}{0}
\renewcommand{\theequation}{\thesection.\arabic{equation}}

Since the seminal work of Mandelbrot and Van Ness \cite{MV68} fractional stochastic processes and fields have gained a lot of attention in the probabilistic and statistical literature. One of the key results of \cite{MV68} was the introduction of the 
\textit{fractional Brownian motion}, the most famous fractional stochastic process. The authors show that, up to scaling,   fractional Brownian motion with Hurst parameter $H\in (0,1)$ is the unique self-similar  centered Gaussian process with  stationary increments.  
When the Gaussianity assumption is replaced by  $\al$-stable distributions
the aforementioned class becomes much richer: Rosinski \cite{R95} shows that it consists of (mixed) \textit{moving average $\al$-stable processes},  \textit{$\al$-stable harmonizable processes} and processes of a third kind (which lack an explicit description). The most well-known representative of the first class is the \textit{linear fractional stable motion} (lfsm), whose probabilistic and statistical properties have been studied in numerous articles. We refer to \cite{BHP19,BLP17,BP17,BPT20, LP18,PT03,PTA07} for a detailed exposition of limit theorems
for functionals of lfsm, and to \cite{AH12,DI17,GLT15,LP20, LP22,MOP20} for statistical estimation of lfsm and related processes. 

In the paper \cite{CM89} the authors introduced the $\al$-stable  harmonizable fractional motion, defined as
\bee \label{model}
X_t= \int_{\R} \frac{\exp(i ts) -1}{i s} |s|^{1-H -1/\al}\, dL_s,
\eee
where $i:=\sqrt{-1}$, $H\in (0,1)$ is the Hurst parameter and 
$L_t:=L_t^1 + i L_t^2$, $t\in \R$, denotes an  isotropic
complex-valued $\al$-stable L\'evy motion with $\alpha\in (0,2)$.  The processes $(X_t)_{t\in \R}$ is self-similar with index $H$, i.e.\ 
$(X_t)_{t\in \R} \eqschw  (a^{-H}X_{at})_{t\in \R} $, and possesses stationary increments. Probabilistic properties of model \eqref{model} and their extensions have been studied in numerous articles, see e.g.\ \cite{AX16,BCI02, CM89, DS11,KM91} among others. However, to the best of our knowledge, asymptotic theory for statistics of harmonizable fractional $\al$-stable processes has not yet being investigated in the literature. In contrast to lfsm, harmonizable fractional $\al$-stable processes are not ergodic, which indicates potential difficulties when studying limit theorems.

This paper aims to derive  asymptotic results for the quadratic statistic 
\begin{equation}\label{ljsdlfjlsjh}
\sum_{j=0}^{n-1} \|X_j -X_{j-1} \|^2,
\end{equation}
and our main statement  is the following theorem:

\begin{theo} \label{main}
Let $(X_t)_{t\geq 0}$  be the  $\al$-stable harmonizable fractional  motion introduced in \eqref{model} with parameters $\alpha\in (0,2)$ and $H\in (0,1)$. 
\begin{enumerate} 
\item [(i)] \label{sldsdfsdjflsj} As $n\to\infty$, 
\bee \label{lln}
\frac{1}{n}\sum_{j=0}^{n-1} \|X_j -X_{j-1} \|^2 \toop U:=2 \int_{\R} |s|^{-2H-2/\al}
\left(1-\cos(s) \right) d([L^1]_s+[L^2]_s),
\eee
where $[L^k]$ denotes the quadratic variation of $L^k$ for $k=1,2$.

\item [(ii)] For $H>1/2$
and $\alpha(1-H)<1/2$ we obtain the weak convergence
\begin{align}  \label{clt} 
{}& n^{2-2H}\left(\frac{1}{n}\sum_{j=0}^{n-1} \|X_j -X_{j-1} \|^2 - U \right) \\
{}& \qquad \quad
\schw 
2{\normalfont \RE}\Big( \int_{\R^2} \frac{1-\exp(i(s-u)) }{i(s-u)} |su|^{1-H-1/\al} 1_{\{u<s\}} \,d\overline{L}_u \,dL_s\Big).
\end{align} 
\end{enumerate} 
\end{theo}

\begin{rem} 
(A)
The result of Theorem~\ref{main}(i) is in sharp contrast to the setting of short memory processes. To make this comparison concrete, let us consider a sequence  $(Y_i)_{i\in \N}$ of independent $\alpha$-stable isotropic  random variables with $\alpha\in (0,2)$. According to the
 Stable  Central Limit Theorem, cf.\  \cite[Theorem~5.25]{JHJ},  
\begin{equation}
\frac{1}{n^{2/\alpha}} \sum_{j=0}^{n-1} \| Y_j\|^2\schw Z,  
\end{equation}
where $Z$ is an $(\alpha/2)$-stable random variable. That is, 
for 
independent $\alpha$-stable random variables, the quadratic variation 
$\sum_{j=0}^n \|Y_j\|^2$ is of the order $n^{\alpha/2}$, whereas 
 $\sum_{j=0}^{n-1} \| X_j-X_{j-1}\|^2$ 
is of the order $n$, cf.\ Theorem~\ref{main}(i). This difference may be viewed 
as an indication  of the strong dependence between  the random variables 
$(X_j-X_{j-1})_{j\in \N}$.

\smallskip
\noindent 
(B) We believe that the randomness of the limit in \eqref{lln} 
is connected to the non-ergodicity of the process $X$. 

\smallskip
\noindent 
(C) The weak convergence result of \eqref{clt} is similar in spirit to the Gaussian case. Indeed, the quadratic variation of the fractional Brownian motion with Hurst parameter $H$ is asymptotically distributed according to the Rosenblatt random variable when $H\in (3/4,1)$ with convergence rate $n^{2H-2}$; the classical Rosenblatt random variable possesses the same representation as the right hand side of \eqref{clt} where $L$ is replaced by a complex-valued Gaussian measure. We also recall that there is no convergence regime of the form \eqref{clt} in the setting of lfsm 
(cf.\ \cite{BHP19,BLP17}).

\smallskip
\noindent 
(D) The proof of Theorem~\ref{main} heavily relies on the quadratic form of the statistic. Extensions to more general functional forms and statistical estimation of the process $X$ are therefore far from obvious. \hfill $\diamond$
\end{rem}

The rest of the paper is structured as follows. In Chapter \ref{sec2}, we discuss the integration theory and a dominated convergence result we will use in our proof. Chapter \ref{sec3} is devoted to the proof of Theorem \ref{main}.

\subsection*{Notation}
Throughout the paper all random variables are defined on a  probability space $(\Omega,\mathcal F,\P)$.
By $\toop$ we denote convergence in probability, and $\schw$ denotes convergence in distribution. 
For a complex number $x\in \mathbb{C}$, $\overline{x}$ denotes its complex conjugate
and $\|x\|$ its norm. Furthermore, we denote by $\text{Re}(x)$ (resp. $\text{Im}(x)$) the real (resp.\ imaginary) part of $x$.   

\section{Some background} \label{sec2} \label{sec2}
\setcounter{equation}{0}
\renewcommand{\theequation}{\thesection.\arabic{equation}}

In this section we will collect some important results on 
stochastic integration theory. 
A complex-valued random variable $Z$ is called isotropic if it holds that $\exp(i\te)Z\eqschw Z$ for any $\te \in (0,2\pi)$.  Similarly, a stochastic process $(Z_t)_{t\in \R}$ 
is isotropic  when any linear combination $\sum_{j=1}^d a_j Z_{t_j}$ is isotropic. It is easy to see
that a complex-valued L\'evy process $L$ with characteristic triplet $(0,0,\nu)$, where $\nu$ denotes the L\'evy measure,  is isotropic if and only if 
\bee
\nu \circ (x\mapsto  x\exp(i\te))^{-1}= \nu \qquad \forall \te \in (0,2\pi).
\eee 
In particular, $L^1$ and $L^2$ both have the same symmetric $\al$-stable distribution. For any measurable function 
$f:\R\to \mathbb{C}$ we define
\begin{align}
{}& \int_{\R} f(s) \,dL_s\\
{}& \quad :=\Big( \int_{\R}  \text{Re}(f(s)) \,dL^1_s - \int_\R 
\text{Im}(f(s)) \,dL^2_s\Big) 
+i \Big( \int_{\R}   \text{Re}(f(s)) \,dL^2_s + \int_\R \text{Im}(f(s)) \,dL^1_s\Big),
\label{ljsdfljslj}
\end{align}
whenever the four real-valued integrals on the right-hand side exist.  
We note that the  integral $\int_{\R} f(s) \,dL_s$ is well defined if and only if $\int_{\R} \|f(s)\| \,dL^1_s$ is well defined, which holds if 
$\int_{\R} \|f(s)\|^{\al}\, ds<\infty$.  The next result, which follows by \cite[Theorem 2]{ST94}, gives a criterion for the existence 
of the double integral:

\begin{prop} \label{int-condi} 
Let $\psi:\R\to (0,\infty)$ be a measure function satisfying $\int_\R \psi(s)^\alpha\,ds=1$, and suppose   that $(Z_t)_{t\in \R}$ is a two-dimensional symmetric $\alpha$-stable L\'evy process with $Z_t = (Z^1_t,Z^2_t)$. The double integral $\int_{\R^2} f(u,v) \,dZ^1_u\,dZ^2_v$  of a measurable function $f:\R^2\to \R$ vanishing outside $\{(u,v)\in \R:u<v\}$ 
exists when the following condition holds: 
\begin{equation}\label{slfjlsfhg}
\int_\R \int_\R |f(s,t)|^\alpha \Big[1+ \log_+\Big(\frac{|f(s,t)|}{\psi(s) \psi(t)}\Big)\Big] \,ds\,dt<\infty,
\end{equation}
where $\log_+(x):=\log(x)$ for $x\geq 1$ and $\log_+(x):=0$ for $x<1$. 
\end{prop} 

We remark that \cite[Theorem 2]{ST94} is only stated in the case where $Z^1=Z^2$; however, the extension to $Z^1\neq Z^2$ follows directly using a polarization techniques similar to the polarization equality for real numbers $xy= ((x+y)^2-(x-y)^2)/4$. Proposition~\ref{slfjlsfhg}  extends directly to complex-valued functions and processes as well.

%
%
%

\begin{rem} \label{rem1}
Let $g:\R\to\CC$ is measurable function with 
$\|g\|_{\al}^{\al}:= \int_\R \|g(x)\|^{\al} dx<\infty$ and set  $f(u,v)=\overline{g}(u)g(v) 1_{\{u<v\}}$ for all $u,v\in \R$. Then we have the following statements. 

\noindent
(A) The function $f$ satisfies  condition \eqref{slfjlsfhg}, and hence the 
 double integral $\int_{\R^2}f(u,v) \,d\overline{L}_u\, dL_v$ is well defined, 
 cf.\ Proposition~\ref{int-condi}.  Indeed, \eqref{slfjlsfhg} follows by choosing 
   $\psi(s)= \|g(s)\|/\|g\|_\alpha$. 

\noindent
(B) The following identity holds
  \begin{equation}\label{lsjdfljs} 
 \Big\| \int_\R g(s)\,dL_s\Big\|^2
 = 2\text{Re}\Big(\int_{\R^2} \overline{g}(u)g(s) 1_{\{u<s\}} \,d\overline{L}_u\, dL_s\Big) + \int_\R \|g(s)\|^2
\,d([L^1]_s+[L^2]_s).
 \end{equation}
 To show \eqref{lsjdfljs} we note that 
$Y_t:=\int_{-\infty}^{t} g(u)\, dL_u$ exists for all  $t\in \R \cup \{\infty\}$.
For all $w<t$ in $\R$ we have by integration by parts that 
\begin{align}
Y_t\overline{Y}_t -Y_w\overline{Y}_w= {}& \int_w^t Y_{s-} \,d\overline{Y}_s+ \int_w^t \overline{Y}_{s-}\,dY_s+[Y,\overline Y]_t-[Y,\overline Y]_w\\
 = {}& 2\text{Re}\Big(\int_u^t \overline{Y}_{s-} \,dY_s\Big)+
 [Y,\overline Y]_t-[Y,\overline Y]_w,\label{ljsdlfjhlsdh} 
\end{align}
where $[Z,\tilde Z]$ denotes the quadratic variation of two complex-valued semimartingales $Z$ and $\tilde Z$. 
If we let $t\to\infty$ and $w\to-\infty$ in  \eqref{ljsdlfjhlsdh} we obtain
\begin{equation}\label{dsljfdsljfls}
Y_\infty\overline{Y}_\infty= 2\text{Re}\Big(\int_\R \overline{Y}_{s-} \,dY_s\Big)+
 [Y,\overline Y]_\infty.
\end{equation}
By reconizing the integral on the right-hand side of \eqref{dsljfdsljfls}
as a double integral and calculating the 
quadratic variation $[Y,\overline Y]$ we 
derive the identity 
\begin{equation}\label{ljsdlfjsl} 
\|Y_\infty\|^2=Y_\infty\overline{Y}_\infty= 2\text{Re}\Big(\int_{\R^2} \overline{g}(u)g(v) 1_{\{u<v\}} \,d\overline{L}_u\, dL_v\Big) + \int_\R 
\|g(s)\|^2
\,d([L^1]_s+[L^2]_s),\end{equation}
where we have used that $[L,\overline{L}]_t=[L^1]_t+[L^2]_t$.
Equation~\eqref{ljsdlfjsl} implies \eqref{lsjdfljs}. 
 \hfill $\diamond$
\end{rem}

We will also use the dominated convergence result 
\cite[Theorem~6.2(c)]{KS89}, which reads as follows.

\begin{prop} \label{propdom}
Assume that $(h_n)_{n\geq 1}$, $h$ and $g$ are deterministic measurable functions $\R^2\to\R$ supported on $\{(u,v)\in \R^2\!:u<v\}$  all satisfying condition \eqref{slfjlsfhg}. Suppose  moreover that $(Z_t)_{t\in \R}$ is a two-dimensional symmetric $\alpha$-stable L\'evy process with $Z_t = (Z^1_t,Z^2_t)$.  If $h_n(u,v) \to h(u,v)$ for Lebesgue almost all points $(u,v)\in \R^2$ and $\sup_n \|h_n\| \leq g$, then 
\bee
\int_{\R^2}
h_n(u,v)\, dZ^1_u \,dZ^2_v \toop \int_{\R^2}h(u,v) \,dZ_u^1\, dZ^2_v
 \qquad \text{as } n\to \infty.
\eee 
\end{prop}

\section{Proof of Theorem \ref{main}} \label{sec3} \label{sec3}
\setcounter{equation}{0}
\renewcommand{\theequation}{\thesection.\arabic{equation}}
Throughout this section all positive constants are denoted by $C$, although they may change from line to line. We start with the identity
\bee \label{reprX}
X_j-X_{j-1} = \int_{\R} \exp(ijs) r(s) \,dL_s \qquad \text{with} \qquad 
r(s):= \frac{1-\exp(-is)}{is} |s|^{1-H-1/\al}. 
\eee
In particular, we have that 
\bee
\|r(s)\| \leq C\left(|s|^{1-H-1/\al} 1_{\{|s|\leq 1\}} + |s|^{-H-1/\al} 1_{\{|s|> 1\}}  \right),
\eee 
which implies the statement
\bee \label{rfinite}
\|r\|_{\al}<\infty.
\eee
By using the representations \eqref{reprX} and \eqref{lsjdfljs} on the function $g:\R\to\CC$ given by $g(s)=\exp(ijs) r(s)$ we deduce the identity
\bee\label{ljsldfjsdljf}
\|X_j-X_{j-1}\|^2 = V_j+U,
\eee 
where 
\begin{align} 
U&:= 2 \int_{\R} |s|^{-2H-2/\al}
\left(1-\cos(s) \right) d([L^1]_s+[L^2]_s) \qquad \text{and} \nonumber \\[1.5 ex]
\label{decomp} V_j &:= 2\text{Re}\Big(\int_{\R^2} \exp(ij(s-u)) r(s) \overline{r}(u) 1_{\{u<s\}} \,d\overline{L}_u\, dL_s\Big) . 
\end{align}

\subsection{Proof of convergence in \eqref{lln}}
By the decomposition \eqref{ljsldfjsdljf} we have 
\begin{equation}
\frac{1}{n}\sum_{j=0}^{n-1} \|X_j-X_{j-1}\|^2 = U +\frac{1}{n} \sum_{j=0}^{n-1} V_j,
\end{equation}
and thus we only need to prove that 
$n^{-1} \sum_{j=0}^{n-1} V_j \toop 0$. We first note that
\bee \label{ident}
\frac{1}{n} \sum_{j=0}^{n-1} \exp(ij(s-u)) =  \frac{1 - \exp(in(s-u))}{n(1 - \exp(i(s-u)))}\to 0
\qquad \text{as } n\to \infty,
\eee
for any $s-u \not \in \{2\pi m:~m\in \Z\}$, using the geometric series.  On the other hand, we have
\bee
\| \exp(ij(s-u)) r(s) r(u) \| \leq \|r(s) \| \|r(u) \|.
\eee
In view of \eqref{rfinite} and Remark~\ref{rem1}(A), we can apply Proposition \ref{propdom} to conclude that 
\bee
n^{-1} \sum_{j=0}^{n-1} V_j \toop 0, 
\eee
which  implies \eqref{lln}.

\subsection{Proof of weak convergence}
We will divide the proof into several steps. We write $g_n(x):= (1 - \exp(inx))/(1 - \exp(ix))$
and define the following complex-valued functions
\begin{align}
h_n(s,u) &:= n^{1-2H}g_n(s-u) r(s) \overline{r}(u) 1_{\{u<s\}}, \\[1.5 ex]
h(s,u)&:=   \frac{1 - \exp(i(s-u))}{i(s-u)} |su|^{1-H-1/\al} 1_{\{u<s\}}.
\end{align}
By the representation \eqref{ljsldfjsdljf} we have 
\begin{equation}
n^{2-2H}\left(\frac{1}{n}\sum_{j=0}^{n-1} \|X_j -X_{j-1} \|^2 - U \right) 
= n^{1-2H} \sum_{j=0}^{n-1} V_j
=  2\RE\Big( 
\int_{\R^2} h_n(s,u)\, d\overline{L}_u \,dL_s\Big),
\label{ljsdlfjlsdjf}
\end{equation}
where the last equality follows from the identity in \eqref{ident} and the definition of $h_n$. Thus, to 
show the limit theorem \eqref{clt} it is enough to show the convergence 
\begin{equation}\label{jlsjdflhlsh}
\int_{\R^2} h_n(s,u)\, d\overline{L}_u \,dL_s\schw
 \int_{\R^2} h(s,u)\, d\overline{L}_u \,dL_s.
\end{equation}
The following subsections will be denoted to the proof of \eqref{jlsjdflhlsh}, and to this aim we define $$\gamma:=1-H-1/\alpha.$$
%
%
%

\subsubsection{Existence of the limit}
In this subsection we will show that the limiting integral $\int_{\R^2} h(s,u)\, d\overline{L}_u \,dL_s$ from \eqref{jlsjdflhlsh} is well-defined 
by proving  that $h$ satisfies condition \eqref{slfjlsfhg}. 
Set  
\bee \label{defpsi}
\psi(s) := c\left(|s|^{-r}\1_{\{|s|>1\}}+\1_{\{|s|\leq 1\}}\right)
\eee 
for all $s\in \R$, where $r$ is any real number satisfying $r>\alpha^{-1}$ and $c=c_r>0$ is chosen such that  $\int_\R \psi(s)^\alpha\,ds=1$. To show condition \eqref{slfjlsfhg} with respect to the function $\psi$ we will use the following lemma.

\begin{lem} \label{lembound}
Let $f:\R^2 \to \mathbb{C}$ be a complex-valued function satisfying the inequality
\bee \label{condonf}
\|f(s,u)\|\leq C |su|^{-r_1} \left(1_{[s-1,s)}(u) + |s-u|^{-r_2} 1_{(-\infty,s-1)}(u)  \right),
\eee
where the numbers $r_1,r_2$ satisfy the conditions $r_1\in (1/2,1)$ and $2r_1+r_2>2$. 
Then it holds that 
\bee
\int_{\R^2} \|f(s,u)\| \,du\,ds<\infty.
\eee
\end{lem}

\begin{proof}
According to condition \eqref{condonf} we have that $\|f(s,u)\|\leq f_1(s,u) + f_2(s,u)$ and hence
we need to show the integrability of the functions $f_1$ and $f_2$. We have that
\[
\int_{s-1}^s |u|^{-r_1} du \leq C \left(1_{[-1,1]}(s) + s^{-r_1} 1_{[-1,1]^c}(s) \right)
\]
since $r_1\in (1/2,1)$. Hence, we conclude
\[
\int_{\R^2} f_1(s,u) \,du \,ds \leq C \int_{\R} |s|^{-r_1} \left( 1_{[-1,1]}(s) + 
s^{-r_1} 1_{[-1,1]^c}(s)\right) ds
\]
and the latter is finite as $r_1\in (1/2,1)$. 

Now we turn our attention to the function $f_2$. We use the substitution $s=xu$ to obtain 
the bound
\begin{align}
\int_{u+1}^{\infty} |s|^{-r_1} |s-u|^{-r_2} ds &= |u|^{1-r_1-r_2} 
\int_{\frac{u+1}{u}}^{\infty} |x|^{-r_1} |x-1|^{-r_2} dx \\[1.5 ex]
&\leq C  \left(1_{[-1,1]}(u) + \left(|u|^{-r_1}+|u|^{1-r_1-r_2} \right)1_{[-1,1]^c}(u)\right)
\end{align}
Indeed, the latter integral is finite since $r_1+r_2>2r_1+r_2 -1>1$. Thus, we deduce
\[
\int_{\R^2} f_2(s,u) \,ds \,du \leq C \int_{\R} |u|^{-r_1} \left( 1_{[-1,1]}(u) 
+ \left(|u|^{-r_1}+|u|^{1-r_1-r_2} \right)1_{[-1,1]^c}(u) \right) ds < \infty
\]
due to assumptions $r_1\in (1/2,1)$ and $2r_1+r_2>2$. This completes the proof.
\end{proof}

\noindent
Now, we proceed with the proof of existence. First of all, we observe that the function 
$\|h(s,u)\|^{\alpha}$ satisfies the condition  \eqref{condonf} with 
\[
r_1 = \alpha \gamma \qquad \text{and} \qquad r_2= \alpha.
\]
It is easy to see that assumptions $r_1\in (1/2,1)$ and $2r_1+r_2>2$ are satisfied due to 
$H>1/2$ and $\alpha(1-H)<1/2$. On the other hand, for any $\ep>0$ we readily deduce that 
\bee \label{logest}
\log_+\left(\frac{\|h(s,u)\|}{\psi(s)\psi(u)}\right)\leq C\left(|s|^{\ep} + |s|^{-\ep} \right)
\left(|u|^{\ep} + |u|^{-\ep} \right).
\eee
When we choose $\ep>0$ small enough, we conclude that the function 
\[
\|h(s,u)\|^{\alpha} \left(1+ \log_+\left(\frac{\|h(s,u)\|}{\psi(s)\psi(u)}\right)\right)
\]
also satisfies the conditions of Lemma \ref{lembound}. This proves the existence of the double integral 
$\int_{\R^2} h(s,u) \,d\overline{L}_u \,dL_s$, cf.\ Proposition~\ref{int-condi}.

\subsubsection{Main decomposition}
Now, we show the weak limit theorem in \eqref{clt}. For this purpose, we introduce the sets 
\[
\mathcal{Z}:=\{2\pi j:~j\geq 0 \} \qquad \text{and} \qquad \mathcal{Z}_{\de}:=\{x \in \R_+:~|x-\mathcal{Z}|<\de \}
\]
for some fixed $\de\in (0,1)$. For $x\in \R_+$ we also use the notation $[x]:=2\pi j$ where 
$2\pi j\in \mathcal{Z}$ is the closest number to $x$. We obtain the following estimate for the 
function $g_n$:
\begin{align}
\|g_n(x)\| &\leq C \left(\left\| \frac{1 - \exp(in(x-[x]))}{i(x-[x])} \right\| 1_{\mathcal{Z}_{\de}} (x) + \de^{-1}1_{\mathcal{Z}^c_{\de}} (x) \right) \nonumber \\[1.5 ex]
 &\leq C \left( n 1_{\mathcal{Z}_{\de} \cap \mathcal{Z}_{1/n}} + \frac{1}{|x-[x]|} 1_{\mathcal{Z}_{\de} \cap \mathcal{Z}_{1/n}^c} + \de^{-1}1_{\mathcal{Z}^c_{\de}} (x)  \right)
\nonumber \\[1.5 ex]
\label{gnbound} &\leq C \left(n^{1-a} |x-[x]|^{-a} 1_{\mathcal{Z}_{\de}} + \de^{-1}1_{\mathcal{Z}^c_{\de}} (x) \right)
\end{align}
for any $a\in (0,1)$.
We  decomposition the double integral on the right-hand side of 
\eqref{ljsdlfjlsdjf} as follows
\bee \label{ljsdlfhjsldh} 
\int_{\R^2} h_n(s,u)\, d\overline{L}_u \,dL_s = Z_n^{(1)} +
Z_n^{(2)} + Z_n^{(3)},
\eee
with 
\begin{align}
Z_n^{(1)} &:= \int_{(-\de,\de)^2} h_n(s,u) \,d\overline{L}_u\, dL_s \\[1.5 ex]
Z_n^{(2)} &:= \int_{(-\de,\de)^c \times (-\de,\de)^c} h_n(s,u) \,d\overline{L}_u\, dL_s \\[1.5 ex]
Z_n^{(3)} &:= \int_{\big((-\de,\de)^c \times (-\de,\de) \big)\cup \big((-\de,\de) \times (-\de,\de)^c\big)} h_n(s,u)\, d\overline{L}_u\, dL_s.
\end{align}
In the following, we will prove that $Z_n^{(1)}$ is the dominating term and show its convergence, while we will prove that $
\|Z_n^{(2)}\| + \|Z_n^{(3)}\|\toop 0$. The following identity in distribution 
\bee \label{idenindi}
\int_A f(s,u) \,d\overline{L}_u\, dL_s \eqschw n^{-2/\al} \int_{nA} f(s/n,u/n)\, d\overline{L}_u\,dL_s,
\eee
where $nA:=\{(ns,nu)\!:~ (s,u)\in A\}$, will be used in our proof.   The identity \eqref{idenindi}  follows  from the self-similarity of $L$; more precisely, from the  fact that  $(L_{t/n})_{t\in \R}$ equals  $(n^{-1/\alpha} L_t)_{t\in \R}$ in finite dimensional distributions. 

\subsubsection{Term $Z_n^{(1)}$}\label{ljsdflj-1}
We apply the distributional identity \eqref{idenindi} to deduce that
\bee
Z_n^{(1)} \eqschw n^{-2/\al} \int_{(-\de n,\de n)^2} h_n(s/n,u/n) \,d\overline{L}_u\, dL_s.
\eee
For any fixed $(s,u)$ with $s-u \not \in \mathcal{Z}$ we obtain the convergence
\bee
\lim_{n\to \infty} n^{\ga} r(s/n) = |s|^{\ga} \qquad \text{and} \qquad 
\lim_{n\to \infty} n^{-1} g_n(s-u) =  \frac{1-\exp(i(s-u))}{i(s-u)}.
\eee
Hence, we conclude that 
\bee
\lim_{n\to \infty} n^{-2/\al} h_n(s/n,u/n) = h(s,u). 
\eee
On the other hand, we have that 
\bee
\|n^{-2/\al} h_n(s/n,u/n)\|^{\al} \leq C |su|^{ \gamma}
\left( \1_{\{s-1\leq u<s\}}+ \1_{\{u<s-1\}} |s-u|^{-\alpha}\right)=:f(s,u).
\eee
As in the proof of existence we deduce that the function $f$ satisfies the condition  \eqref{slfjlsfhg}.  Thus, 
\bee 
Z_n^{(1)} \schw 
\int_{\R^2} \frac{1-\exp(i(s-u)) }{i(s-u)} |su|^{1-H-1/\al} 1_{\{u<s\}} \,d\overline{L}_u\, dL_s
\eee
due to Proposition \ref{propdom}.

\subsubsection{Term $Z_n^{(2)}$}\label{ljsdflj-2}

We will use the approximation \eqref{gnbound}. First of all, 
we have that 
\begin{equation}\label{ljsdlfjlsj}
\|h_n(s,u)\| 1_{ \mathcal{Z}_{\de}^c}(s-u )\leq C \de^{-1} n^{1-2H} |su|^{\ga -1} 1_{\{u<s\}} .
\end{equation}
Note that  $n^{1-2H}\to 0$
since $H>1/2$. On the other hand, the function $f(s,u):= |su|^{\ga -1} 1_{\{u<s\}}$ satisfies the condition
\[
\int_{(-\de,\de)^c \times (-\de,\de)^c} f(s,u)^{\alpha}  \left(1+ \log_+\left(\frac{f(s,u)}{\psi(s)\psi(u)}\right)\right) du \,ds<\infty
\] 
where we use an estimate of the type \eqref{logest} for the log term. 
For all $a\in (0,1)$ we have
\begin{align}\label{ljsdlfjslj}
\|h_n(s,u)\|1_{\mathcal{Z}_{\de}} (s-u) 
 &\leq C n^{2-2H -a} |su|^{\ga -1} \left|(s-u) - [s-u]\right|^{-a} 1_{\{u<s\}}.
\end{align}
It is easy to show that the function $(s,u)\mapsto  |su|^{\ga -1} |(s-u) - [s-u]|^{-a} 1_{((-\de,\de)^c)^2}(s,u) 1_{\{u<s\}}$ satisfies the condition \eqref{slfjlsfhg} when $a\al < 1$. Choosing $a=1/\al -\ep$ we deduce that
\[
n^{2-2H -a} = n^{2-2H -1/\al +\ep} \to 0 \qquad \text{as } n\to \infty,
\]
if we choose $\ep>0$ small enough since $\al(1-H)<1/2$. We, therefore, conclude that  there exists a measurable function $\phi$ satisfying \eqref{slfjlsfhg} such that $|h_n(u,v)|\leq \phi(u,v)$ for all  
$(u,v)\in (-\de,\de)^c \times (-\de,\de)^c$. Since $h_n(u,v)\to 0$ Lebesgue almost surely,  cf.\ \eqref{ljsdlfjlsj}--\eqref{ljsdlfjslj}, Proposition~\ref{propdom} yields $Z_n^{(2)} \toop 0$.

\subsubsection{Term $Z_n^{(3)}$}\label{ljsdflj-3} 
The proof is similar to the one in Subsection~\ref{ljsdflj-2}. For simplicity we only consider the integration region $(-\de,\de)\times (-\de,\de)^c$; the  
$(-\de,\de)^c\times (-\de,\de)$ case is  analogue. For all $(u,s)\in (-\de,\de)\times (-\de,\de)^c$:
\begin{equation}\label{ljlsjdfljs}
\|h_n(s,u)\| 1_{ \mathcal{Z}_{\de}^c}(s-u )\leq C \de^{-1} n^{1-2H} |u|^{\ga}|s|^{\ga -1} 1_{\{u<s\}} .
\end{equation}
We have  $n^{1-2H}\to 0$ and the function $f(s,u):=|u|^{\ga } |s|^{\ga -1} 
1_{(-\de,\de)}(u)1_{(-\de,\de)^c}(s)1_{\{u<s\}}$ satisfies the condition  \eqref{slfjlsfhg}.
On the other hand, we  have 
\begin{equation}\label{ljlsdjfljsdlfjs}
\|h_n(s,u)\|1_{\mathcal{Z}_{\de}} (s-u) 
 \leq C n^{2-2H -a} |u|^{\ga } |s|^{\ga -1}  \left|(s-u) - [s-u]\right|^{-a} 1_{\{u<s\}}.
\end{equation}
The bounds \eqref{ljlsjdfljs} and \eqref{ljlsdjfljsdlfjs} imply the existence of 
 a measurable function $f$ satisfying \eqref{slfjlsfhg} such that $|h_n(u,v)|\leq f(u,v)$ for all $(u,v)\in (-\de,\de)\times (-\de,\de)^c$. Since a similar bound holds on 
 $(-\de,\de)^c\times (-\de,\de)$ we  conclude, cf.\ 
  Proposition~\ref{propdom}, that   $Z_n^{(3)} \toop 0$ since $h_n\to 0$ 
  (cf.\  \eqref{ljlsjdfljs}--\eqref{ljlsdjfljsdlfjs}).

By the decomposition \eqref{ljsdlfhjsldh} and the 
convergence results of Subsections~\ref{ljsdflj-1}--\ref{ljsdflj-3} 
we obtain the limit theorem \eqref{jlsjdflhlsh}. This completes the proof of 
\eqref{clt} and of Theorem~\ref{main}. 
\qed

\bibliographystyle{chicago}
 
\end{document}